\documentclass{gtpart}     % Basic GT/GTM/AGT style
\usepackage[latin1]{inputenc}
\usepackage{amsfonts}
\title{Short geodesics in hyperbolic 3-manifolds}

\author{William Breslin}
\givenname{William}
\surname{Breslin}
\address{Department of Mathematics, University of Michigan, Ann Arbor, MI 48109}
\email{breslin@umich.edu}
\urladdr{http://www-personal.umich.edu/~breslin/index.html}

\arxivreference{0912.3496}
\arxivpassword{xycvv}

\volumenumber{}
\issuenumber{}
\publicationyear{}
\papernumber{}
\startpage{}
\endpage{}
\doi{}
\MR{}
\Zbl{}
\received{}
\revised{}
\accepted{}
\published{}
\publishedonline{}
\proposed{}
\seconded{}
\corresponding{}
\editor{}
\version{}

\newtheorem{lemma}{Lemma}
\newtheorem{claim}{Claim}
\newtheorem{thm}{Theorem}

\newcommand{\bd}{\partial}

\newcommand{\h}{\mathbb{H}}
\newcommand{\T}{\tilde{T}}

\newcommand{\tS}{\tilde{S}}

\begin{document}

\begin{abstract}    For each $g \ge 2$, we prove existence of a computable constant $\epsilon(g) > 0$ such that if $S$ is a strongly irreducible Heegaard surface of genus $g$ in a complete hyperbolic 3-manifold $M$ and $\gamma$ is a simple  geodesic of length less than $\epsilon(g)$ in $M$, then $\gamma$ is isotopic into $S$.
\end{abstract}

\maketitle

%%%%%%%%%%%%%%%%%%%%   Start of main body of article

\section{Introduction}\label{intro}

Let $M$ be a hyperbolic 3-manifold and let $S$ be a surface embedded in $M$.  A simple closed curve $\gamma$ in $M$ is said to be \textit{unknotted} with respect to $S$ if $\gamma$ can be isotoped into $S$.  A finite collection $\Gamma = \{ \gamma_1, ... ,\gamma_n \}$ of simple closed curves is \textit{unlinked} with respect to $S$ if there is a collection of disjoint embedded surfaces $S_1, ... , S_n$ which are isotopic to $S$ and with $\gamma_i \subset S_i$ for all $i$.  One can ask if short geodesics are unknotted or unlinked with respect to fibers, Heegaard surfaces, or leaves of a foliation.

It follows from work of Otal \cite{Otal} that short geodesics in a hyperbolic mapping torus are unlinked with respect to the fiber, where ``short" depends only on the genus of the fiber.

\begin{thm}[Otal]\label{Otal}
For every $g$ there is a constant $\epsilon > 0$ such that the following holds: If $M$ is a closed 3-dimensional hyperbolic mapping torus with genus $g$ fiber and $\Gamma$ is the collection of simple closed geodesics in $M$ which are shorter than $\epsilon$, then $\Gamma$ is unlinked with respect to $S$.
\end{thm}

In an unpublished paper \cite{Souto}, Souto proved that short geodesics in hyperbolic compression bodies are unlinked with respect to the boundary of the compression body.

\begin{thm}[Souto \cite{Souto}]\label{comp}
If $\bar{N}$ is a compression body then there is a constant $\epsilon > 0$ which depends only on $\chi(\bar{N})$ such that for every complete hyperbolic metric on the interior $N$ of $\bar{N}$ we have: every finite collection of simple geodesics which are shorter than $\epsilon$ is unlinked with respect to $\bd\bar{N}$.
\end{thm}

In the same paper Souto sketched a proof that short geodesics in hyperbolic 3-manifolds are unlinked with respect to a strongly irreducible Heegaard surface:

\begin{thm}[Souto \cite{Souto}]\label{shorts}
For every $g$ there is a constant $\epsilon > 0$ such that the following holds: if $M$ is a closed hyperbolic 3-manifold, $S$ is a strongly irreducible Heegaard surface of genus $g$ in $M$, and $\Gamma$ is the collection of simple closed geodesics in $M$ which are shorter than $\epsilon$, then $\Gamma$ is unlinked with respect to $S$.
\end{thm}

The constants from Theorem \ref{Otal} and Theorem \ref{comp} are known to be computable.  However Souto only sketches a proof of Theorem \ref{shorts} and does not produce an explicit constant.  In this paper, we develop a new approach to Theorem \ref{shorts}.  Our proof is more topological than geometric and is more elementary than Souto's proof.  Moreover, we prove existence of a \textit{computable} constant $\epsilon$ (depending only on the genus $g$) such that primitive geodesics of length less than $\epsilon$ are unknotted with respect to a strongly irreducible Heegaard surface of genus $g$:

\begin{thm}\label{short}
For each $g \ge 2$ there exists a computable constant $\epsilon := \epsilon(g) > 0$ such that if $S$ is a strongly irreducible Heegaard surface of genus $g$ in a complete hyperbolic 3-manifold $M$ and $\gamma$ is a simple geodesic of length less than $\epsilon$ in $M$, then $\gamma$
is isotopic into $S$.
\end{thm}

See the remark after the proof of Lemma \ref{lem1} for a description of the constant $\epsilon(g)$.\\

The proof of Theorem \ref{short} uses three main tools: bounded area sweepouts provided by work of Pitts and Rubinstein \cite{PittsRub}, an argument using the Rubinstein-Scharlemann graphic similar to an argument of Johnson \cite{johnson} used to prove that spines of strongly irreducible Heegaard splittings are locally unknotted, and a lemma of Schultens from \cite{Schultens}. In section \ref{annulus}, we use bounded area sweepouts and the Rubinstein-Scharlemann graphic to prove existence of an embedded annulus connecting a Heegaard surface to the boundary of a Margulis tube around a short geodesic. First, using a bounded area sweepout and the fact that Margulis tubes around very short geodesics are very fat, we show that the intersection of one of the sweepout surfaces with the Margulis tube around a short geodesic contains a simple loop which is homotopic to a power of the short geodesic.  This is the content of Lemma \ref{lem1} and is the only step which uses a geometric argument.  Next, we use the Rubinstein-Scharlemann graphic and Lemma \ref{lem1} to show that there exists an embedded annulus connecting a Heegaard surface to the boundary of a Margulis tube around the short geodesic.  This is the content of Lemma \ref{annulus}.
In section ref{isotopy}, a thin position argument of Schultens is used to show that the existence of the annulus provided by Lemma \ref{annulus} implies that the short geodesic is isotopic into a Heegaard surface.\\

\noindent\textit{Definitions.}  Let $M$ be a closed connected orientable 3-manifold.  Let $S$ be a closed connected orientable surface embedded in $M$ which bounds  handlebodies $H_1$ and $H_2$ on either side.  We call $(S,H_1 ,H_2)$ a \textit{Heegaard splitting} of $M$.  A Heegaard splitting is \textit{weakly reducible} if there are properly embedded essential disks in $H_1$, $H_2$ whose boundaries are disjoint.  A Heegaard splitting is \textit{strongly irreducible} if it is not weakly reducible.

\section{Finding an annulus}\label{annulus}

Let $S$ be a strongly irreducible Heegaard surface of genus $g \ge 2$ in a complete hyperbolic 3-manifold $M$.
The goal of this section is to prove the following Lemma.

\begin{lemma}\label{annulus}
There exists a computable constant $\epsilon := \epsilon(g)$ such that if $\gamma$ is a simple closed geodesic of length less than $\epsilon$ in $M$, then there is a regular neighborhood $N$ of $\gamma$ and an embedded annulus in $M \setminus N$ with boundary $\alpha \cup \alpha '$, where $\alpha$ is a simple essential non-meridinal loop in the boundary of $N$, and $\alpha '$ is contained in a surface isotopic to $S$.
\end{lemma}

\noindent\textit{Sweepouts and Pitts-Rubinstein.} A sweepout of of a 3-manifold $M$ with a Heegaard surface $S$ is a smooth degree one map $f: S \times [0,1] \rightarrow M$ such that  $S_t := f(S \times \{ t \}  )$ is a surface isotopic to $S$ for each $t\in (0,1)$ and $f(S \times \{ 0 \} )$, $f(S \times \{ 1 \} )$ are spines of the handlebodies bounded by $S_{1/2}$.  By work of Pitts and Rubinstein \cite{PittsRub}, there is a constant $A(g)$ such that if $M$ has a genus $g$ strongly irreducible Heegaard surface $S$ then there exists a sweepout $f: S \times [0,1] \rightarrow M$ of $M$ such that $area(S \times \{ t \} ) \le A(g)$ for each $t \in [0,1]$.  We can use $A(g) = 4\pi(g-1) + \delta$ for any $\delta > 0$. \\

We will use bounded area sweepouts to prove the following Lemma.

\begin{lemma}\label{lem1}
There exists a computable constant $\epsilon := \epsilon(g) > 0$ such that the following holds: if $\gamma$ is a simple closed geodesic of length less than $\epsilon$ in $M$ and $T$ is a Margulis tube about $\gamma$, then $S$ may be isotoped in $M$ so that $S \cap T$ contains a simple loop which is essential in $T$.
\end{lemma}

\begin{proof}
Let $f: S \times [0,1] \rightarrow M$ be a sweepout of $M$ with $area(S \times \{ t \} ) \le A := A(g)$ for each $t \in [0,1]$.
Let $\epsilon = \epsilon(g) > 0$ be so small that the Margulis tube about any geodesic in a complete hyperbolic 3-manifold of length at most epsilon has totally geodesic meridian disks with area at least $3A$.  Note that the constant $\epsilon$ is computable.  Let $\gamma$ be a geodesic of length at most $\epsilon$.

The set $M \setminus S_{1/2}$ is a union of two handlebodies.  Let $H_{1/2}$ and $W_{1/2}$ be the closures of these handlebodies.  For $t \in (0,1)$, let $H_t$ be the closure of the component of $M \setminus S_t$ which changes into $H_{1/2}$ as $S_t$ isotopes to $S_{1/2}$ and let $W_t$ be the closure of the other component of $M \setminus S_t$.  For $t \in (0,1)$ near $0$, one of the handlebodies, say $H_t$, is a small neighborhood of a spine.  Since $f$ has degree one, the handlebody $W_t$ is a small neighborhood of a spine for $t \in (0,1)$ near $1$.

If $S_t \cap T$ contains a loop which is essential in $T$ for some $t\in [0,1]$, then we are done proving Lemma \ref{lem1}.  Assume that $S_t \cap T$ does not contain a loop which is essential in $T$ for any $t\in [0,1]$.\\

\begin{claim}\label{claim1}
For each $t\in[0,1]$ either the interior of  $W_t \cap T$ or the interior of $H_t \cap T$ contains a loop which is essential in $T$.\\
\end{claim}

\noindent\textit{Proof.}  
We are assuming that $S_t \cap T$ does not contain a loop which is essential in $T$.  If $S_t \cap T $ is empty then we are done with Claim \ref{claim1}, so we will assume that $S_t \cap T$ is nonempty.
Thus we have that $S_t \cap T$ is nonempty and that any loop in $S_t \cap T$ is trivial in $T$.  Let $\T$ be a lift of $T$ to the universal cover $\h^3$ of $M$.  Since any loop in $S_t \cap T$ is trivial in $T$, there is a lift $\tS$ of $S_t \cap T$ contained in $\T$ which is homeomorphic to $S_t \cap T$. Let $\tS _0$ be a connected component of $\tS$.  Since $\T$ is a ball, $\tS _0$ must separate $\T$.   We claim that $\T \setminus \tS_0$  contains a component which has compact closure which does not separate the ends of $\T$.

Let $D$ be a totally geodesic meridian disk in $\tilde{T}$ such that $D$ is orthogonal to $\bd\tilde{T}$.  Consider the projection $p : \tilde{T} \rightarrow D$ of $\tilde{T}$ to $D$ along lines equidistant from the geodesic core of $\tilde{T}$.  The area of $S_t$ is at least the area of $p(\tilde{S})$.  Since the area of $S_t$ is less than the area of $D$, the interior of $D$ contains a point $x$ which is not in $p(\tilde{S})$.  The preimage $p^{-1} (x)$ of $x$ is disjoint from $\tS _0$ and therefore contained in one component of $\T \setminus \tS _0$.  Thus $\tS_0$ does not separate the ends of $\T$
Also, $\tS _0$ is compact, hence contained in some compact subset $K$ of $\T$.  The complement $\T \setminus K$ of $K$ in $\T$ is contained in the component of $\T \setminus \tS _0$ which contains $p^{-1} (x)$.  Thus the other component of $\T \setminus \tS _0$ is contained in the compact set $K$ and therefore has compact closure.  
We have shown that every component of $\tS$ splits off a connected piece of $\T$ which does not separate the ends of $\T$ and which has compact closure. 

Let $D_1$ and $D_2$ be distinct meridian disks in $\T$ which project to the same disk in $T$.  
There is one component of $\T \setminus (D_1 \cup D_2)$ whose closure is compact.  Let $\mathcal{F}$ be the closure of this component.   Since $S_t$ is compact, there are finitely many lifts $\tS _1 , ... , \tS _k$ of $S_t \cap T$ which intersect the compact set $\mathcal{F}$.  For each $i = 1, ... ,k$, $\tS _i$ splits a piece from $\T$ which has compact closure so  each component of $\tS_1\cup \cdots \cup \tS_k$ splits off a piece of $\T$ which does not separate the ends of $\T$ and which has compact closure.  Thus the set $\T \setminus (\tS _1 \cup \cdots \cup \tS _k)$ contains a connected component which intersects both $D_1$ and $D_2$.  Therefore we can find an arc in $\T \setminus (\tS _1 \cup \cdots \cup \tS _k)$ with endpoints in $D_1$ and $D_2$. We have shown that the complete pre-image of $S_t \cap T$ does not separate the two ends of $\tilde{T}$ and thus some component $\tilde{C}$ of the pre-image of $T \setminus S_t$ in $\T$ is non-compact.  This non-compact component $\tilde{C} \subset \T$ projects to a set $C$ in $T \setminus S_t$.   Since $\tilde{C}$ is non-compact, the set $C$ has nontrivial image in $\pi_1(T)$ under the map induced by inclusion.  Thus there is a loop contained in $C$ (which is contained in either  the interior of $W_t \cap T$ or the interior of $H_t \cap T$) which is essential in $T$. \hfill$\Box$(Claim \ref{claim1}) \\

\noindent The following Claim will complete the proof of Lemma \ref{lem1}.\\

\begin{claim}\label{claim2}
The Heegaard surface $S_{\beta}$ contains a simple loop in $S_\beta \cap T$ which is essential in $T$ for some $\beta \in (0,1)$.
\end{claim}

\noindent\textit{Proof.}  Since $H_t$ is a small neighborhood of a spine for $t \in (0,1)$ near $0$, the interior of $W_t \cap T$ contains a loop which is essential in $T$ for $t \in (0,1)$ near $0$.  Fix $\delta > 0$ so that the interior of $W_{\delta} \cap T$ contains a loop which is essential in $T$.  If $H_{\delta} \cap T$ contains a loop which is essential in $T$, then let $\beta = \delta$.
If $H_{\delta} \cap T$ does not contain a loop which is essential in $T$, then let $\sigma = \inf \{t \in (\delta ,1) | H_t \cap T$ contains a loop which is essential in $T\}$. Note that $\sigma$ exists since $W_{t}$ bounds a very small neighborhood of a graph for $t$ near $1$.  

If the interior of $H_\sigma \cap T$ contains a loop which is essential in $T$, then $H_t \cap T$ contains a loop which is essential in $T$ for $t$ near $\sigma$, contradicting the definition of $\sigma$.
Thus the interior of $H_{\sigma} \cap T$ does not contain a loop which is essential in $T$ and therefore Claim \ref{claim1} implies that the interior of $W_{\sigma} \cap T$ contains a loop which is essential in $T$.  So $W_t \cap T$ contains a loop which is essential in $T$ for $t$ near $\sigma$. This implies that for some $\beta > \sigma$ near $\sigma$, both $H_\beta \cap T$ and $W_\beta \cap T$ contain loops $l_W \subset W$ and $l_H \subset H$  which are essential in $T$.
For some natural numbers $m,n$ we have $(l_W)^m$ is homotopic in $T$ to $(l_H)^n$.  Thus there is an immersed annulus $A$ in $T$ with boundary components equal to $(l_W)^m$ and $(l_H)^n$.  If $A \cap S_\beta$ does not contain any loops which are essential in $A$, then there is an arc in $A$ with endpoints in $l_W$ and $l_H$ which is disjoint from $S_\beta$.  This contradicts the fact that $l_W$ and $l_H$ are in different components of $T \setminus S_\beta$.  Thus some loop $\alpha$ in $A \cap S_\beta$ is essential in $A$.  Since the boundary components of $A$ are essential in $T$, the loop $\alpha$ must be essential in $T$.  Thus $S_\beta$ contains  a loop which is essential in $T$, so $S_\beta$ contains  a simple loop which is essential in $T$. \hfill $\Box$(Claim \ref{claim2})

\end{proof}

\noindent\textbf{Remark.}  The constant $\epsilon(g)$ in Lemma \ref{lem1} is the constant we will use in Theorem \ref{short}.  From the proof of Lemma \ref{lem1}, $\epsilon(g)$ should be so small that a meridian disk in a Margulis tube around a closed geodesic of length less than $\epsilon(g)$ has area greater than $A(g) = 4\pi(g-1)$.  The area of a totally geodesic disk of radius $r$ in hyperbolic space is $2\pi(\cosh(r) - 1)$.  Thus we want the radius of the Margulis tube to be greater than $\operatorname{arccosh}(2g-1)$.  Meyerhoff \cite{meyerhoff} proved that the radius $r$ of a Margulis tube around a closed geodesic of length $l$ less than $0.107$ satisfies $\sinh^2 (r) = 1/2(\frac{\sqrt{1-2k}}{k} - 1)$ where $k = \cosh(\sqrt{\frac{4\pi l}{\sqrt{3}}}) - 1$.   Thus choosing $\epsilon(g)$ so that $\operatorname{arcsinh(\sqrt{1/2(\frac{\sqrt{1-2k}}{k} - 1)}}) > \operatorname{arccosh}(2g-1)$, where $k = \cosh(\sqrt{\frac{4\pi \epsilon(g)}{\sqrt{3}}}) - 1$ will suffice.\\

By Lemma \ref{lem1}, we may isotope $S$ so that $S \cap T$ contains a simple loop $\omega$ which is essential in $T$.  We will show that we may isotope $S$ so that $S \cap \bd T$ contains a simple loop which is essential in $T$ by using an argument of Johnson used in \cite{johnson} to prove that spines of strongly irreducible Heegaard splittings are locally unknotted.  The idea is to use the Rubinstein-Scharlemann graphic to show that some sweepout surface $S'$ intersects $\bd T$ in loops which are inessential in $S'$.  Since $S$ and $S'$ cobound a thickened copy of $S$, there is an embedded annulus with one boundary component equal to $\omega$ and the other boundary component disjoint from $T$.  The intersection of this annulus with $\bd T$ must contain a simple loop which is essential in the annulus, providing us with an embedded annulus to isotope $\omega$ into $\bd T$.

We will now define a new sweepout of $M$.  Assume that we have isotoped $S$ so that $S \cap T$ contains a simple loop $\omega$ which is essential in $T$.  The Heegaard surface $S$ splits $M$ into two handlebodies $H_1$ and $H_2$.
Let $f : M \rightarrow [-1,1]$ be a smooth function such that $f^{-1} (-1)$ is a spine of $H_1$, $f^{-1} (1)$ is a spine of $H_2$, $f^{-1} (t)$ is a surface isotopic to $S$ for each $t\in(-1,1)$, and $f^{-1}(0) = S$ (i.e., the map $f$ provides a sweepout of $M$ by disjoint Heegaard surfaces, one of which is $S$).  For each $t\in(-1,1)$, let $S_t = f^{-1} (t)$.  Note that the surfaces in this sweepout do not necessarily have area bounded in terms of $g$.\\

\begin{lemma}\label{disk}
Let $Y$ be a solid torus in $M$.  One of the following holds:\\
(1) $S_t \cap \bd Y$ contains a loop which is essential in $\bd Y$ and does not bound a meridian disk in $Y$ for some $t\in (-1,1)$, or \\
(2) For all $t\in (-1,1)$, if $S_t \cap \bd Y$ contains a loop which is essential in $S_t$, then $S_t \cap \bd Y$ contains a loop which bounds a properly embedded, essential disk in one of the handlebodies bounded by $S_t$.
\end{lemma}

\begin{proof}
Assume that $S_t \cap \bd Y$ does not contain a loop which is both essential  in $\bd Y$ and does not bound a meridian disk in $Y$ for all $t\in (-1,1)$.  Thus any loop in $S_t \cap \bd Y$ is either trivial in $\bd Y$ or a meridian for $Y$.  In particular, any loop in $S_t \cap \bd Y$ bounds a properly embedded disk in $M$.  Scharlemann's No Nesting Lemma (see \cite{Scharlemann}) implies that any loop which is essential in $S_t$ and bounds a disk in $M$ bounds a properly embedded, essential disk in one of the handlebodies bounded by $S_t$.  Thus if $S_t \cap \bd Y$ contains a loop which is essential in $S_t$, then $S_t \cap \bd Y$ contains a loop which bounds a properly embedded, essential disk in one of the handlebodies bounded by $S_t$.
\end{proof}

\noindent\textit{The Rubinstein-Scharlemann graphic.} Let $g : T \rightarrow [0,1]$ be a smooth function such that $g^{-1} (0) = \gamma$ and $g^{-1} (1) = \bd T$ and $g^{-1} (t)$ is a surface isotopic to $\bd T$ for each $t$ in $(0,1)$.  
For each $t \in [-1,1]$ let $g_t = g|_{S_t \cap T}$.  We say a function is \textit{near-Morse} if there is a single degenerate critical point or there are two non-degenerate critical points at the same level.   By work of Cerf \cite{cerf}, we can isotope $f$ and $g$ so that $g_t$ is a Morse function for all but finitely many $t$ and $g_t$ is near-Morse for the remaining values of $t$.   The Rubinstein-Scharlemann graphic, $G$, is the set of points $(t,s_t ) \in [-1,1] \times [0,1]$ such that $s_t$ is a critical value of the function $g_t$.   
Rubinstein and Scharlemann originally used the graphic to compare two sweepouts by Heegaard surfaces (see \cite{rub-scharl}).  We are using a sweepout of the 3-manifold be Heegaard surfaces and a sweepout of a solid torus by tori.  The properties of the graphic we need to use follow from the same arguments Rubinstein and Scharlemann use in their paper.
Rubinstein and Scharlemann showed that this set of points is a graph with vertices of valence 2 and 4 in the interior of $[-1,1] \times [0,1]$ and valence 1 and 2 in the edges.  A valence-2 vertex at $(t,s_t )$ occurs when the map $g_t$ has a degenerate critical point.  A valence-4 vertex at $(t,s_t )$ occurs when $g_t$ has critical points at the same level.  We will use the fact that if $(t_1 , s_1 )$ and $(t_2 ,s_2 )$ are in the same component of $[-1,1] \times [0,1] \setminus G$, then the surface $S_{t_1}$ is isotopic to $S_{t_2}$ via an isotopy which takes the loops in $g_{t_1}^{-1}(s_1)$ to the loops in $g_{t_2}^{-1}(s_2)$.     
See  \cite{rub-scharl} or \cite{johnson} for more on the Rubinstein-Scharlemann graphic.

\begin{lemma}
Either $S_t \cap \bd T$ contains a loop which is essential in $\bd T$ and does not bound a meridian disk in $T$ for some $t\in (-1,1)$ or there is a $\sigma \in(-1,1)$ such that $S_{\sigma} \cap T$ does not contain an essential loop of $S_{\sigma}$.
\end{lemma}

\begin{proof}
Assume that $S_t \cap \bd T$ does not contain a loop which is an essential non-meridinal loop in in $\bd T$ for any $t\in (-1,1)$.
Suppose, for contradiction, that $S_{t} \cap T$ contains an essential loop of $S_{t}$ for each $t\in(-1,1)$.  For each $t\in(-1,1)$, let $C_t$ be a component of $S_t \cap T$ that contains a loop which is essential in $S_t$.
First we will show that this implies that for each $t\in(-1,1)$ such that $g_t$ is a Morse function, there exists a simple loop in $S_t \cap \bd T$ which bounds a properly embedded, essential disk in one of the handlebodies bounded by $S_t$.
Suppose $g_t$ is a Morse function and suppose that there is no simple loop in $S_t \cap \bd T$ which is essential in $S_t$.  In particular, each loop in $C_t \cap \bd T$ bounds a disk in $S_t$.  We may isotope $S_t$ to eliminate any of these disks which are disjoint from the interior of $T$ and we still have that $C_t$ contains a loop which is essential in $S_t$.  If $C_t \cap \bd T$ is still non-empty, then $C_t$ must be a disk since we have eliminated any disks bounded by loops in $C_t \cap \bd T$ which are disjoint from the interior of $T$.  This contradicts our assumption that $C_t$ contains a loop which is essential in $S_t$.  If $C_t \cap \bd T$ is empty after eliminating disks outside $T$, then we have that $S_t$ is isotopic into $T$, giving us a contradiction.  Thus, for each $t\in(-1,1)$, if $g_t$ is a Morse function then there exists a simple loop in $S_t \cap \bd T$ which is essential in $S_t$. By Lemma \ref{disk} this implies that for each $t\in(-1,1)$ such that $g_t$ is a Morse function, there exists a simple loop in $S_t \cap \bd T$ which bounds a properly embedded essential disk in one of the handlebodies bounded by $S_t$.

For each component $A$ of $([-1,1] \times [0,1]) \setminus G$, choose a point $(t_A , s_A)$ in the interior of $A$.  If one of the loops in $g_{t_A}^{-1}(s_A)$ bounds a properly embedded essential disk in the handlebody $f^{-1}[-1,t_A]$, then label the component $A$ with a $1$.  If one of the loops in $g_{t_A}^{-1}(s_A)$ bounds a properly embedded essential disk in the handlebody $f^{-1}[t_A ,1]$, then label the component $A$ with a $2$.
If some component $A$ of  $([-1,1) \times [0,1]) \setminus G$ has more than one label, then some loop in $g_{t_A}^{-1}(s_A)$ bounds disks in both $f^{-1}[-1,t_A]$ and $f^{-1}[t_A ,1]$ implying that the Heegaard splitting is weakly reducible.

We have shown that for each $t\in(-1,1)$ such that $g_t$ is a Morse function, there exists a simple loop in $S_t \cap \bd T = g_t^{-1} (1)$ which bounds a disk in one of the handlebodies bounded by $S_t$.  In other words, any component of $([-1,1] \times [0,1]) \setminus G$ which meets $[-1,1] \times \{ 1 \}$ is labeled.

For $t$ near $-1$, the vertical line $\{t \} \times [0,1]$ must intersect a component labeled 1, because $S_t$ is near the spine of $f^{-1} [-1,t]$.  For $t$ near $1$, the vertical line $\{t \} \times [0,1]$ must intersect a component labeled 2, because $S_t$ is near the spine of $f^{-1} [t,1]$.  Thus there must be adjacent segments in $[-1,1] \times \{ 1 \} \setminus G$ which have different labels.  Since $G$ cannot contain any vertical segments, this implies that some vertical line $\{t \} \times [0,1]$ intersects two labeled components with different labels.  Then for some $s_1, s_2 \in [0,1]$, there is a loop in $g_{t}^{-1}(s_1)$ which bounds a disk in $f^{-1}[-1,t_A]$ and there is a loop in $g_{t}^{-1}(s_2)$ which bounds a disk in $f^{-1}[t_A ,1]$.  Since  $g_{t}^{-1}(s_1)$ and $g_{t}^{-1}(s_2)$ are disjoint, the Heegaard splitting is weakly reducible, giving a contradiction.

We have shown that there is a $\sigma \in(-1,1)$ such that $S_{\sigma} \cap T$ does not contain an essential curve of $S_{\sigma}$.
\end{proof}

\noindent\textit{Proof of Lemma \ref{annulus}.}
We will show that we may isotope $S$ so that $S \cap \bd T$ contains a loop which is essential and non-meridinal in $\bd T$.
Let $f : M \rightarrow [-1,1]$ be a smooth function such that $f^{-1} (-1)$ is a spine of $H_1$, $f^{-1} (1)$ is a spine of $H_2$, $f^{-1} (t)$ is a surface isotopic to $S$ for each $t\in(-1,1)$, and $f^{-1} (0) = S$.
Suppose that $S_t \cap \bd T$ does not contain a loop which is essential and non-meridinal in $\bd T$ for any $t\in (-1,1)$.
We have isotoped $S = S_0$ so that $S \cap T$ contains a simple loop $\omega$ which is essential in $T$. We have also shown that there is a $\sigma \in (-1,1)$ such that $S_{\sigma} \cap T$ does not contain an essential loop of $S_{\sigma}$ and therefore each loop in $S_{\sigma} \cap \bd T$ bounds a disk in $S_{\sigma}$.  Let $A$ be an annulus embedded in $M$ with one boundary component equal to $\omega$ and the other boundary component $l$ contained in $S_{\sigma}$.  This annulus exists because $S$ and $S_\sigma$ bound a surface times interval. Any loop in $S_\sigma \cap \bd T$ must bound a disk in $S_\sigma$ since $S_{\sigma} \cap T$ does not contain a loop which is essential in $S_{\sigma}$.  
We may isotope $A$ so that $l$ is disjoint from $S_{\sigma} \cap \bd T$ since loops in $S_{\sigma} \cap \bd T$ bound disks in $S_\sigma$.  The loop $l$ is now either disjoint from $T$ or contained in $T$ since it is disjoint from $\bd T$.  Since $l$ is essential in the 3-manifold $M$, $l$ must be essential in $S_\sigma$.  Since $S_{\sigma} \cap T$ does not contain an essential loop of $S_{\sigma}$, we must have that $l$ is contained in $M \setminus T$.
We may isotope $\omega$ by a small isotopy so that it is contained in the interior of $T$.  
We now have an annulus $A$ embedded in $M$ such that $\bd A = \omega \cup l$ with $\omega \subset T$ and $l \subset (M\setminus T)$.  Thus there must be a simple loop $l'$ in $A \cap \bd T$ which is essential in $A$.   The embedded annulus bounded by $l \cup l'$ can be used to isotope $S_\sigma$ so that $S_\sigma \cap T$ contains a loop which is essential and non-meridinal in $\bd T$.

We have shown that we may isotope $S$ so that $S \cap \bd T$ contains a loop which is essential and non-meridinal in $\bd T$.  Now we can let $N$ be a regular neighborhood of $\gamma$ contained in $T$ and disjoint from $\bd T$.  Since $S \cap \bd T$ contains a loop which is essential and non-meridinal in $\bd T$, there is an embedded annulus in $M \setminus N$ with boundary $\alpha \cup \alpha '$, where $\alpha$ is a simple essential non-meridinal loop in the boundary of $N$, and $\alpha '$ is contained $S$.
\hfill $\Box$

\section{Finding an isotopy}\label{isotopy}

\begin{lemma}\label{schultens}
Let $\gamma$ be a simple loop in a 3-manifold $M$. Let $M = H_{1} \cup_{S} H_{2}$ be a Heegaard splitting of $M$.  Let $\alpha$ be a simple essential non-meridinal loop in the boundary of a regular neighborhood $N$ of $\gamma$.  If there is an embedded annulus $A$ in $M$ disjoint from the interior of $N$ with boundary $\alpha \cup \alpha '$ where $\alpha ' \subset S$, then $\gamma$ is isotopic into $S$.
\end{lemma}

The proof of Lemma \ref{schultens} is a thin position argument used by Schultens in \cite{Schultens} to show that exceptional fibers in Seifert manifolds are isotopic into a Heegaard surface (Lemma 4.1 in \cite{Schultens}).   See Johnson \cite{jessenotes} for a use of this argument to classify genus-one Heegaard splittings of lens spaces.  The only adjustment needed for our result is that the loop $\gamma$ can be put into thin position while keeping the annulus $A$ embedded.

\begin{proof}
Let $f : M \rightarrow [-1,1]$ be a smooth function such that $f^{-1} (-1)$ is a spine of $H_1$, $f^{-1} (1)$ is a spine of $H_2$, and $f^{-1} (t)$ is a surface isotopic to $S$ for each $t\in(-1,1)$.
Let $g: A \rightarrow M$ be an embedding of an annulus $A$ into $M$ with $\bd A = \alpha \cup \alpha '$ such that $g(\alpha)$ is an essential non-meridinal loop in the boundary of the regular neighborhood $N$ of $\gamma$, $g(\alpha) \cap N = \emptyset$, and $g(\alpha ') \subset S$.  Extend $g$ to an immersion (also called $g$) such that $g(\alpha) = \gamma^n$ and $g|_{A\setminus\alpha}$ is an embedding.

If $n=1$, then we are done, so assume that $n \ge 2$.  Let $h = f \circ g : A \rightarrow [-1,1]$.  We may assume that after a small isotopy $h|_{A\setminus\alpha}$ is a Morse function and that $h|_{\alpha}$ has no degenerate critical points.  Since $\alpha '$ is mapped to the level surface $S$, we may also assume that the singular foliation $\mathcal{F}$ of $A$ by level sets of $h$  consists of parallel circles in a neighborhood of $\alpha '$.  Isotope $\gamma$ and $g$ so that $\gamma$ is disjoint from the cores of the handlebodies in $M\setminus S$, while keeping $g|_{A\setminus\alpha}$ an embedding.

The singular foliation $\mathcal{F}$ contains an \textit{essential saddle} if it contains a saddle singularity $x$ such that the four arcs in the level set containing $x$ emanating from $x$ end on $\alpha$.  If $x$ is an essential singularity, then $x$ and the arcs emanating from $x$ cut off three disks from $A$.  If none of the disks cut off by an essential saddle $x$ and the arcs emanating from $x$ contain an essential saddle, then we call $x$ an \textit{outermost essential singularity}.  
If $\beta$ is an outermost level arc in the foliation $\mathcal{F}$ of $A$, then it cuts off a disk $D$.  Call $D$ an \textit{upper disk} if $h(x) > h(\beta)$ for all $x \in D\setminus\beta$ and call $D$ a \textit{lower disk} if $h(x) < h(\beta)$ for all $x \in D\setminus\beta$.   

Suppose that a level set of $\mathcal{F}$ contains an outermost essential saddle which 
splits off an upper disk $D_u$ between two lower disks $D_1$ and $D_2$ such that $g(D_1 \cap \alpha ) \cap g(D_2 \cap \alpha ) = \emptyset$. Note that since $\mathcal{F}$ consists of parallel circles in a neighborhood of $\alpha '$, we have that the disks $D_u$, $D_1$, and $D_2$ are disjoint from $\alpha'$. We may isotope $\gamma$ and the map $g$ to push $D_u$ below $D_1$ and $D_2$.  We may similarly isotope $\gamma$ and $g$ to eliminate a lower disk $D_l$ between two upper disks $D_3$ and $D_4$ such that $g(D_3 \cap \alpha ) \cap g(D_4 \cap \alpha ) = \emptyset$.  Isotope $\gamma$ and $g$ until no such triples of upper and lower disks exist.  We may choose these 
 isotopies so that $g|_{A\setminus\alpha}$ is still an embedding since the upper and lower disks are disjoint from a neighborhood of $\alpha'$.
If $\gamma \subset S'$ for a level set $S'$ of $f$, then we are done.  Otherwise the map $f|_{\gamma}$ has at least two critical points so that $h|_{\alpha}$ has at least $4$ critical points.

By Proposition 3.1 of \cite{Schultens}, after an arbitrarily small isotopy of the map $g$ near $\alpha$, $\mathcal{F}$ contains an \textit{outermost essential saddle}.
An outermost essential saddle must cut off either an upper disk between two lower disks or a lower disk between two upper disks. Without loss of generality, assume it cuts off an upper disk $D_u$ between two lower disks $D_1$ and $D_2$.  We have isotoped $\gamma$ and $g$ so that
we cannot have $g(D_1 \cap \alpha ) \cap g(D_2 \cap \alpha ) = \emptyset$.  Thus we must have $g(D_1 \cap \alpha ) = g(D_2 \cap \alpha )$.  This implies that $g( (D_1 \cap \alpha) \cup (D_u \cap \alpha) ) = \gamma$, so that $g(D_1 \cup D_u )$ provides a disk to
isotope $\gamma$ to $g( \bd (D_1 \cup D_u) - \alpha)$ which is contained in a level surface of $f$.  In other words, we may isotope $\gamma$ into a Heegaard surface parallel to $S$.
\end{proof}

\noindent{\textit{Proof of Theorem \ref{short}.}}
The proof of Theorem \ref{short} now follows from Lemma \ref{annulus} and Lemma \ref{schultens}. \hfill $\Box$\\

\textbf{Acknowledgements.}  This work was partially supported by NSF grant DMS-0135345 and NSF RTG grant 0602191. The author would like to thank Joel Hass for suggesting the use of bounded area sweepouts to prove Theorem \ref{short}.  The author would also like to thank Juan Souto, Jesse Johnson, and Jennifer Schultens for helpful conversations, and the referee for helpful suggestions.

%%%%%%%%%%%%%%%%%%%%   End of main body of article
%
%                             References
%
%   BiBTeX users uncomment the following line:
%
\bibliographystyle{gtart}

\bibliography{tri}

%\begin{thebibliography}

%\end{thebibliography}

\end{document}